\documentclass[11pt]{amsart}

\usepackage{amscd}
\usepackage{amsmath, amssymb}
\usepackage{amsfonts}
\usepackage[all]{xy}
\newcommand{\de}{\partial}

\newcommand{\Ric}{\mathrm{Ric}}
\newcommand{\ov}[1]{\overline{#1}}

\newcommand{\ti}[1]{\tilde{#1}}

\newcommand{\ve}{\varepsilon}

\renewcommand{\leq}{\leqslant}
\renewcommand{\geq}{\geqslant}

\newtheorem{theorem}{Theorem}[section]
\newtheorem{lemma}[theorem]{Lemma}

\theoremstyle{definition}
\newtheorem{rk}[theorem]{Remark}

\title[Fibered Calabi-Yau manifolds without singular fibers]{Triviality of fibered Calabi-Yau manifolds without singular fibers}
\author[V. Tosatti]{Valentino Tosatti$^{*}$}
\thanks{$^{*}$Supported in part by a Sloan Research Fellowship and NSF grants DMS-1236969 and DMS-1308988.}
 \address{Department of Mathematics, Northwestern University, 2033 Sheridan Road, Evanston, IL 60201}
  \email{tosatti@math.northwestern.edu}
  \author[Y. Zhang]{Yuguang Zhang$^{\dagger}$}
  \thanks{$^{\dagger}$Supported in part  by NSFC-11271015.}
\address{Mathematical Sciences Center,  Tsinghua University,  Beijing 100084, P.R.China.}
\email{yuguangzhang76@yahoo.com}

\begin{document}
\begin{abstract}
In this note we show that if a compact K\"ahler manifold with trivial canonical bundle is the total space of a holomorphic fibration  without singular fibers, then the fibration is a holomorphic fiber bundle.
In the algebraic case, the fibration becomes trivial after a finite base change.
\end{abstract}
\maketitle

\section{Introduction}
Let $X$ be a compact K\"ahler manifold with $c_1(X)=0$ in $H^2(X,\mathbb{R})$, which we will call a Calabi-Yau manifold. Thanks to a fundamental theorem of Yau \cite{Y}, the class of Calabi-Yau manifolds is precisely the class of compact K\"ahler manifolds which admit Ricci-flat K\"ahler metrics. Furthermore, $X$ Calabi-Yau implies that the canonical bundle $K_X$ is holomorphically torsion \cite{Bo}.
Assume that there is a holomorphic submersion
with connected fibers $f:X\to Y$ to a compact K\"ahler manifold $Y$. Let $n=\dim X$, $m=\dim Y$ and assume that $0<m<n$. Then every fiber $X_y=f^{-1}(y)$ is a
Calabi-Yau manifold of dimension $n-m$. In other words, $X$ is Calabi-Yau and is also
the total space of a family of Calabi-Yau manifolds over $Y$.

In the case when $f$ is allowed to have singular fibers, this is the setup of the first author's paper \cite{To}, where it is proved that certain families of Ricci-flat K\"ahler metrics on $X$ collapse to $Y$ in a weak sense. More recently, Gross and the authors \cite{GTZ} proved that if the smooth fibers are tori then the collapsing happens in the $C^\infty$ topology, and in \cite{GTZ2} we identified the Gromov-Hausdorff limit when $Y$ is a curve.

The question which we study in this paper, which arose from these works, is whether one can classify such fibrations when no singular fibers are allowed. For such manifolds the smooth collapsing was proved by Fine \cite{Fi}. Here we show that these fibrations are very special. Our first result is:

\begin{theorem}\label{main}
If $f:X\to Y$ is a holomorphic submersion with connected fibers between projective manifolds with $K_X\cong\mathcal{O}_X$, then
there is a finite unramified covering $\ti{Y}\to Y$ with $K_{\ti{Y}}\cong\mathcal{O}_{\ti{Y}}$ such
that the pullback family $X\times_Y\ti{Y}\to \ti{Y}$ is trivial, i.e. it is
biholomorphic to the product family $\ti{Y}\times F\to\ti{Y}$, where $F$ is a projective manifold with trivial canonical bundle.
\end{theorem}
In particular, this implies that the canonical bundle of $Y$ is holomorphically torsion, which is not a priori clear.
As a simple consequence we have:
\begin{theorem}\label{main1}
If $f:X\to Y$ is a holomorphic submersion with connected fibers between projective manifolds with $X$ Calabi-Yau, then there are projective manifolds
$B, F$ with trivial canonical bundles, and   finite unramified coverings $g:B\times F\to X$ and $h:B\to Y$ such that the diagram
\begin{equation}\label{diag}
\begin{CD}
 B\times F @>{g}>>X\\
@V{\pi_B}VV @V{f}VV \\
 B @>{h}>>Y
\end{CD}
\end{equation}
commutes.
\end{theorem}

If we drop the projectivity assumption, then the statement of Theorem \ref{main} is false (see Remark \ref{rk1} below), and has to be modified as follows:

\begin{theorem}\label{main2}
If $f:X\to Y$ is a holomorphic submersion with connected fibers between compact K\"ahler manifolds with $K_X\cong\mathcal{O}_X$, then
$f$ is a holomorphic fiber bundle with fiber $F$ a compact K\"ahler manifold with $K_F\cong\mathcal{O}_F$ and with base $Y$ Calabi-Yau.
If furthermore  either $b_1(F)=0$ or $F$ is a torus and $b_1(Y)=0$, then there is a finite unramified covering $\ti{Y}\to Y$ with $K_{\ti{Y}}\cong\mathcal{O}_{\ti{Y}}$ such
that the pullback bundle $X\times_Y\ti{Y}\to \ti{Y}$ is holomorphically trivial.
\end{theorem}

In the literature there are several related results. For example Fang-Lu \cite[Corollary 1.3]{FL} have shown that if the fibers of $f$ are primitive Calabi-Yau
manifolds which satisfy a certain cohomological assumption, then the family is isotrivial (in the sense that all fibers over a Zariski open set of $Y$ are biholomorphic to each other).
This result holds without the assumption that $X$ is Calabi-Yau. More recently, Zhang-Zuo \cite[Corollary 2.5]{ZZ}
that if $Y$ is simply connected and $X$ has trivial canonical bundle and $h^p(X,\mathcal{O}_X)=0$ for $0<p<n$, then the family is isotrivial.
In the setup of Theorem \ref{main}, if one knows in addition that $\kappa(Y)\geq 0$, then an alternative proof of the theorem can be given using the canonical bundle formula \cite{FM} and arguing along the lines of \cite[Theorem 4.8]{A} (see also Remark \ref{rkh} below). If one knows furthermore that $Y$ is Calabi-Yau, then the conclusion that $f$ is a bundle also follows from \cite[Theorem 2]{Bo} (see also \cite[Theorem 3.1]{COP}).

On the other hand we make no assumptions on $Y$, and we first show that $Y$ has a K\"ahler metric with nonnegative Ricci curvature. Using Hodge theory, we get a holomorphic period map from the universal cover of $Y$ to a classifying space $\mathcal{D}$ of polarized real Hodge structures. This space has a Hermitian metric with strictly negative holomorphic sectional curvature in the horizontal directions. Using a version of Yau's Schwarz Lemma \cite{Yau} we show that the period map is constant, and the infinitesimal Torelli theorem for Calabi-Yau manifolds implies that all fibers are biholomorphic, and then the main theorem follows.\\

This note is organized as follows: in section \ref{sect1} we deduce Theorem \ref{main1} from Theorem \ref{main}, and show that to prove Theorem \ref{main} it is enough to prove that all fibers
are biholomorphic. This statement is then proved in section \ref{sect2}. Theorem \ref{main2} is proved in section \ref{sect3}.\\

\noindent {\bf Acknowledgements:} The first named author is grateful to Florin Ambro, Keiji Oguiso and Xiaowei Wang for useful communications. The second named author would like to thank Professors Mark Gross and Zhiqin Lu for some discussions. Both authors are grateful to Andreas H\"oring for communicating the content of Remark \ref{rkh}, to Xiaokui Yang for pointing out an inaccuracy in a previous draft, and to the referee for useful suggestions.

\section{Preliminary results}\label{sect1}

\begin{proof}[Proof of Theorem \ref{main1} assuming Theorem \ref{main}]
Since $K_X$ is holomorphically torsion, there is a finite unramified covering $X'\to X$ with $K_{X'}$ holomorphically trivial.
The composition $f':X'\to X\to Y$ is a holomorphic submersion, but its fibers might be disconnected.
Consider its Stein factorization $X'\overset{f''}{\to}Y'\overset{p}{\to}Y$, which has the properties that $Y'$ is a projective manifold,
$p$ is a finite unramified covering and $f''$ is a holomorphic submersion with connected fibers (see e.g. \cite[Lemma 2.4]{FG2}).
Therefore $f'':X'\to Y'$ satisfies the hypotheses of Theorem \ref{main}, and hence there
is a finite unramified covering $\ti{Y}\to Y'$ such that the pullback family $X'\times_{Y'} \ti{Y}\to \ti{Y}$ is trivial.
We let $B=\ti{Y}$ and let $h:B\to Y$ be the composition $\ti{Y}\to Y'\to Y$. We have a biholomorphism $X'\times_{Y'} \ti{Y}\cong B\times F$, which composed with the map
$X'\times_{Y'}\ti{Y}\to X'\to X$ gives us a finite
unramified covering map $g:B\times F\to X$, and Theorem \ref{main1} follows.
\end{proof}

Note that from Theorem \ref{main} it follows that the fibers $X_y$ of the original family must all be biholomorphic to each other. In fact, it is enough to prove
this assertion to prove Theorem \ref{main}:

\begin{lemma}\label{finite} If the fibers $X_y$ of a family $f:X\to Y$ as in Theorem \ref{main} are all biholomorphic to each other, then the conclusion of Theorem \ref{main} holds.
\end{lemma}
\begin{proof}
Indeed, in this case we can apply the Fischer-Grauert theorem \cite{FG} and conclude that $f$ is a holomorphic fiber bundle,
and then we can apply \cite[Lemma 17]{KL} to get a finite unramified covering $\ti{Y}\to Y$
such that the pullback family $X\times_Y \ti{Y}\to \ti{Y}$ is trivial. Now $X\times_Y \ti{Y}\to X$ is also a finite unramified covering, hence the canonical bundle of
$X\times_Y \ti{Y}$ is holomorphically trivial. But $X\times_Y \ti{Y}\cong \ti{Y}\times F$, and so the canonical bundles of $\ti{Y}$ and $F$ are both trivial.
\end{proof}

In this proof we used \cite[Lemma 17]{KL} which uses crucially the assumption that $X$ is projective. In fact, \cite[Lemma 17]{KL} is false if $X$ is only compact K\"ahler, see Remark \ref{rk1}. However, as we will see in section \ref{sect3}, \cite[Lemma 17]{KL} is true in the K\"ahler case if either $b_1(F)=0$ or $F$ is a torus and $b_1(Y)=0$.

Thus, to prove Theorem \ref{main}, we are reduced to showing that all the fibers are biholomorphic. In the case when $\dim Y=\dim X-1$ the
 proof is very simple: in this case the fibers $X_y$ are elliptic curves,
which are classified by their $j$-invariant, and mapping $y$ to the $j$-invariant of $X_y$ gives a holomorphic function on $Y$ which must be constant since $Y$ is compact.
We follow a similar approach in higher dimensions, which is necessarily more complicated.

\section{Proof of theorem \ref{main}}\label{sect2}
The main result of this section is the following:
\begin{theorem}\label{triv}
If $f:X\to Y$ is a holomorphic submersion between compact K\"ahler manifolds with $K_X\cong\mathcal{O}_X$, then
$f$ is a holomorphic fiber bundle with fiber $F$ a compact K\"ahler manifold with $K_F\cong\mathcal{O}_F$ and with base $Y$ Calabi-Yau.
\end{theorem}

In particular, since all fibers of a holomorphic fiber bundles are biholomorphic, Theorem \ref{main} follows immediately from Theorem \ref{triv} and Lemma \ref{finite}.

\begin{proof}
Let $\omega_X$ be a K\"ahler metric on $X$, and note that since the canonical bundle $K_{X}$ is trivial, the fibers $X_y=f^{-1}(y)$ also have trivial canonical bundle.
Thanks to \cite[Proposition 4.1]{To} (see also \cite{ST}) there is a K\"ahler metric $\omega$ on $Y$ such that ${\rm Ric}(\omega)=\omega_{WP}\geqslant 0,$ where $\omega_{WP}$ is
a Weil-Petersson type semipositive definite form \cite{FS}. In fact more generally it is true that for $Y$ to admit a K\"ahler metric with nonnegative Ricci curvature, it is enough to assume that $X$ admits such a metric, thanks to a result of Berndtsson (\cite[Theorem 1.2]{Ber} applied with $L=-K_X$).

We now need to use real polarized Hodge structures. These are defined without any projectivity assumption, by a straightforward generalization of Griffiths' original definition in the projective case \cite{Gr2}. Here we follow the discussion in \cite[Section 8.1]{GM} (see also \cite[Definition 1.5]{De}).
Let $P\subset H^{n-m}(X_{y}, \mathbb{C})=:H$ be the primitive cohomology induced by the K\"ahler class $[\omega_X|_{X_y}]$. Let
$H_{\mathbb{R}}=P\cap H^{n-m}(X_{y}, \mathbb{R})$, $h^{p,q}=\dim_{\mathbb{C} } P\cap H^{q}(X_{y}, \Omega^{p}_{X_{y}})$ for $p+q=n-m$, and call
$Q$ the quadratic form on $H$ given by
$$Q(\phi, \psi)=(-1)^{\frac{(n-m)(n-m-1)}{2}} \int_{X_{y}}\phi \wedge \psi,$$
which is called a polarization (even though no integrality assumption is made).

Then the construction by Griffiths \cite[Section 8]{Gr3} gives a classifying space $\mathcal{D}$ for real polarized Hodge structure of type
  $\{ H_{\mathbb{R} }, h^{p,q}, Q \}$, and a well-defined holomorphic period map
  $\mathcal{P}:\ti{Y}\to \mathcal{D}$, where $\ti{Y}$ is the universal cover of $Y$. If $\pi:\ti{Y}\to Y$ is the universal covering map, and $y=\pi(z)$, then we have
  $\mathcal{P}(z)= \{P^{p,q}=P\cap H^{q}(X_{y}, \Omega^{p}_{X_{y}})\}$.
Furthermore, $\mathcal{P}$ is horizontal in the sense that $\mathcal{P}_{*} (T(\ti{Y})) \subset T_{h}(\mathcal{D}),$
where $T_{h}(\mathcal{D})$ is the horizontal subbundle of $T(\mathcal{D})\subset \oplus_{r>0} Hom(P^{p,q},P^{p-r,q+r})$ defined for example in Definition 25 of Chapter I in \cite{Gr}.

There is a Hermitian metric $\omega_{H}$ on $\mathcal{D}$ such that for any unit vector $\xi \in T_{h}(\mathcal{D})$,  its holomorphic sectional curvature satisfies
 \begin{equation}\label{curv}
 R(\xi,\ov{\xi},\xi,\ov{\xi}) \leqslant -A<0,
 \end{equation}
  for a constant $A>0$ (cf. \cite[Theorem 9.1]{GS}, \cite[Theorem 5.16]{De} or Chapter II in \cite{Gr}). In fact, the restriction of the metric $\omega_H$ to any horizontal slice is K\"ahler thanks to \cite{Lu}, i.e. $\de\omega_H(\xi_1,\ov{\xi_2},\xi_3)=0$ for all $\xi_1,\xi_2,\xi_3\in T_{h}(\mathcal{D})$.

If we let $\ti{\omega}=\pi^*\omega$, then $\ti{\omega}$ is a complete K\"ahler metric on $\ti{Y}$ with nonnegative Ricci curvature. We then apply the Schwarz Lemma \ref{schwarz} below to conclude that $\mathcal{P}:\ti{Y}\to\mathcal{D}$ is constant.

 By \cite[Theorem 1.27]{Gr2}, \cite[Theorem 5.3.4]{CMP} or Chapter III in \cite{Gr},  $\mathcal{P}_{*,z}: T_{z} \ti{Y} \to T_{h}(\mathcal{D})_{\mathcal{P}(z)}$
 is a composition of the Kodaira-Spencer map $\rho:T_{z} \ti{Y} \to H^{1}(X_{y}, \Theta_{X_{y}})$, ($y=\pi(z)$),
 and the cup product map $w: H^{1}(X_{y}, \Theta_{X_{y}}) \to T_{h}( \mathcal{D})_{\mathcal{P}(z)}$,
 where $\Theta_{X_{y}}$ is the sheaf of holomorphic vector fields on $X_{y}$.
 Since $X_{y}$ has trivial canonical bundle,  $w$ is injective (cf. Example 5.6.2 in \cite{CMP} or \cite[Proposition 3.6]{Gr2}, \cite{Tj}, \cite{BG}).
 This result is also known as the infinitesimal Torelli theorem for Calabi-Yau manifolds.
 Thus the Kodaira-Spencer map $\rho$ of the pullback family over $\ti{Y}$ is trivial. This implies that the Kodaira-Spencer map of the family over $Y$ is trivial too.
Now the fibers $X_y$ have trivial canonical bundle, and hence $H^1(X_y,\Theta_{X_{y}})\cong H^{1,n-1}(X_y)$ has dimension independent of $y\in Y$.
Therefore thanks to \cite[Theorem 4.6]{Ko} the complex structure of $X_{y}$ does not change when $y$ varies, i.e. all fibers $X_y$ are biholomorphic to each other.
The Fischer-Grauert theorem \cite{FG} implies that $f$ is a holomorphic fiber bundle with fiber $F\cong X_y$ with trivial canonical
bundle. Since all fibers $X_y$ are biholomorphic to each other, the Weil-Petersson form $\omega_{WP}$ of the fibration $f$ vanishes identically. Its cohomology class is $[\omega_{WP}]=c_1(Y)$, hence $Y$ is Calabi-Yau.
\end{proof}

\begin{rk}
The proof we just finished in fact shows the following result: if $f:X\to Y$ is a holomorphic submersion between compact K\"ahler manifolds with
fibers $X_y$ with trivial canonical bundle and base $Y$ which admits a K\"ahler metric with nonnegative Ricci curvature then
$f$ is a holomorphic fiber bundle.
\end{rk}
In the proof of Theorem \ref{triv}, we used the following improved version of Yau's Schwarz Lemma \cite{Yau}, which incorporates an observation of Royden \cite{Ro} (see also \cite{CCL, To1}).
\begin{lemma}[Schwarz Lemma]\label{schwarz} Let $(X,\omega_X)$ be a complete K\"ahler manifold with nonnegative Ricci curvature and $(Y,\omega_Y)$ be  a Hermitian manifold. Let
$f:X\to Y$ be a holomorphic map, and assume that the holomorphic sectional curvature of $\omega_Y$ satisfies
\begin{equation}\label{hypo}
R_Y(\xi,\ov{\xi},\xi,\ov{\xi})\leq -A<0,
\end{equation}
for a constant $A>0$ and for all unit vectors $\xi$ in the image of $f_*:T^{1,0}X\to T^{1,0}Y$, and that $\de\omega_Y(\xi_1,\ov{\xi_2},\xi_3)=0$ for all $\xi_1,\xi_2,\xi_3$ in the image of $f_*$. Then $f$ is constant.
\end{lemma}
\begin{proof}
Let $u=\mathrm{tr}_{g_X}(f^*g_Y)$, which is a smooth nonnegative function on $X$, and vanishes identically if and only if $f$ is constant. Assume that $f$ is not constant, so $\sup_X u>0$ (it could also equal $+\infty$), and let $k>0$ be the maximum rank of $f_*$ over $X$, which also equals the generic rank of $f_*$. Our first goal is to show that 
\begin{equation}\label{key}
\Delta_{g_X}u\geq Au^2,
\end{equation}
holds on all of $X$, where $\Delta_{g_X}=g^{i\ov{j}}\de_i\de_{\ov{j}}$ is the (complex) Laplacian of $g_X$ acting on functions.  For any given $x\in X$, we can find an open neighborhood $U\ni x$ such that $V=f(U)$ is an irreducible analytic subvariety of $Y$, of dimension $k$. The variety $V$ may not be smooth at $y=f(x)$, but $f^{-1}(V_{sing})$ is Zariski closed in $U$. We now take any point $x'\in f^{-1}(V_{reg})$, let $y'=f(x')$ and choose a local unitary coframe $\{\theta^1,\dots,\theta^n\}$ for $g_X$ near $x'$ and a local unitary coframe $\{\ti{\theta}^1,\dots,\ti{\theta}^k\}$ for $g_Y|_{V_{reg}}$ near $y'$.
For any $0\leq\alpha\leq k$ write
$$f^*\ti{\theta}^\alpha=\sum_i f^\alpha_i \theta^i,$$
so $\{f^\alpha_i\}$ are the components of $f_*$ in the chosen coframes. Therefore we have $u(x')=\sum_{\alpha,i}|f^\alpha_i|^2$. Then a direct calculation (see e.g. \cite[(5.9)]{To1}, noting that the Laplacian used there is twice the complex Laplacian) shows that at $x'$ we have
\[\begin{split}
\Delta_{g_X}u&\geq \sum_{\alpha,i,j} \Ric(g_X)_{i\ov{j}}\ov{f^\alpha_i}f^\alpha_j-\sum_{\alpha,\beta,\gamma,\delta,i,k}(R_Y)_{\alpha\ov{\beta}\gamma\ov{\delta}}
f^\alpha_i\ov{f^\beta_i}f^\gamma_k\ov{f^\delta_k}\\
&\geq -\sum_{i,k}R_Y(\xi_i,\ov{\xi_i},\xi_k,\ov{\xi_k}),
\end{split}\]
using the assumption that $\Ric(g_X)\geq 0$, where $\xi_i=f_*(e_i)$ and $\{e_1,\dots,e_n\}$ is the unitary frame for $g$ dual to the coframe $\{\theta^i\}$, and $R_Y$ denotes the curvature tensor of the metric $g_Y|_{V_{reg}}$. By assumption, the metric $g_Y|_{V_{reg}}$ is K\"ahler, and has holomorphic sectional curvature bounded above by $-A$ because of assumption \eqref{hypo} and because holomorphic sectional curvature decreases in submanifolds.
By changing the unitary frames at $x'$ and $y'$ if necessary, we may assume that the vectors $\{\xi_i\}$ are pairwise orthogonal.
Then an inequality of Royden \cite[Lemma, p.552]{Ro} gives that
\[\begin{split}
\sum_{i,k}R_Y(\xi_i,\ov{\xi_i},\xi_k,\ov{\xi_k})&\leq -A \left(\sum_i \|\xi_i\|_{g_Y}^2\right)^2
=-A \left(\sum_{\alpha,i}|f^\alpha_i|^2 \right)^2\\
&=-Au^2.
\end{split}\]
This shows that \eqref{key} holds on $f^{-1}(V_{reg})$, and since this is Zariski open in $U$ and both sides of \eqref{key} are smooth functions, this proves that \eqref{key} holds everywhere.

We now conclude the proof exactly as in \cite{Yau, To1}. The generalized maximum principle of Yau (see e.g. \cite[Proposition 4.1]{To1}) shows that if $v$ is a smooth function on $X$ which is bounded below, then given any $\ve>0$ there is a point $x_\ve\in X$ with $\Delta_{g_X}v(x_\ve)\geq-\ve, \|du\|_{g_X}(x_\ve)\leq\ve$ and $\liminf_{\ve\to 0} v(x_\ve)=\inf_X v$. We apply this to $v=(u+1)^{-1/2}$, so that $\limsup_{\ve\to 0} u(x_\ve)=\sup_X u>0$, and we get
\[\begin{split}
-\ve&\leq \Delta_{g_X}v(x_\ve)=-\frac{1}{2}(u(x_\ve)+1)^{-3/2}\Delta_{g_X}u(x_\ve)+\frac{3}{4}(u(x_\ve)+1)^{-5/2}\|du\|^2_{g_X}(x_\ve)\\
&\leq -\frac{A}{2}(u(x_\ve)+1)^{-3/2}u(x_\ve)^2+\frac{3\ve^2}{4},
\end{split}\]
which holds for all $\ve>0$ such that $u(x_\ve)>0$, so that \eqref{key} holds at $x_\ve$. Taking $\ve$ arbitrarily close to $0$, we conclude that $\sup_X u=0$, which is a contradiction.
\end{proof}

\begin{rk}\label{rkh}
Andreas H\"oring has communicated to us an alternative line of proof of Theorem \ref{triv} in the projective case, which uses rather different ingredients.
Here is a brief sketch of the argument.
If the base $Y$ is not uniruled, then $K_Y$ is pseudoeffective thanks to \cite[Corollary 0.3]{BDPP}, and one can conclude
using the canonical bundle formula \cite{FM} and arguing along the lines of \cite[Theorem 4.8]{A}.

If $Y$ is uniruled, take $\mathbb{P}^1 \subset Y$ a free rational curve (i.e. with globally generated normal
bundle). The pullback family $X' = (X \times_Y \mathbb{P}^1)\to\mathbb{P}^1$ is such that $X'$ is smooth, with nonnegative Kodaira dimension (by adjunction), and
$X'\to\mathbb{P}^1$ is a submersion. This contradicts \cite[Theorem 0.2]{VZ}.
\end{rk}

\section{The K\"ahler case}\label{sect3}
In this section we will prove the following result, which is a substitute for Lemma \ref{finite} with no projectivity assumption:
\begin{theorem}\label{bund}
Let $X,Y$ be compact K\"ahler manifolds and $f:X\to Y$ a holomorphic fiber bundle with base $Y$ and fiber $F$ Calabi-Yau manifolds. If either $b_1(F)=0$ or $F$ is a torus and $b_1(Y)=0$, then there is a finite unramified covering $\ti{Y}\to Y$ such that the pullback bundle to $\ti{Y}$ is holomorphically trivial.
\end{theorem}

Theorem \ref{main2} follows immediately from this result together with Theorem \ref{triv}. In fact, in Theorem \ref{bund} we do not even assume that $X$ is Calabi-Yau.

\begin{rk}\label{rk1}
Theorem \ref{bund}, and hence Lemma \ref{finite} and Theorems \ref{main1} and \ref{main}, are false if $X$ is not projective and $b_1(F)>0$ and $b_1(Y)>0$. Indeed, let $\Lambda$ be the lattice in $\mathbb{C}^2$
spanned by the vectors $(1,0), (0,1), (i,0), (i\sqrt{2},i)$, and let $X=\mathbb{C}^2/\Lambda$. Then the torus
$X$ has algebraic dimension $1$ (see e.g. \cite[p. 50]{EF}), and its algebraic reduction map is $f:X\to Y$ with $Y$ an elliptic curve.
Then $f$ is an elliptic bundle, which is not trivial (otherwise $X\cong Y\times F$ would have algebraic dimension $2$).
Since the algebraic dimension is invariant under finite unramified coverings, the bundle does not become trivial even after pulling back via
a finite unramified covering $\ti{Y}\to Y$.
\end{rk}

\begin{proof}[Proof of Theorem \ref{bund}]
Fix $\omega$ a K\"ahler metric on $X$, let $\mathrm{Aut}^0(F)$ be the connected component of the identity of the automorphism group of $F$, and
let $\mathrm{Aut}(F,[\omega|_F])$ be the group of automorphisms of $F$ which preserve the K\"ahler class $[\omega|_F]$. The structure group of the bundle $f:X\to Y$ reduces
to $\mathrm{Aut}(F,[\omega|_F])$. Thanks to \cite[Proposition 2.2]{Li} or \cite[Theorem 4.8]{Fu} the quotient $\mathrm{Aut}(F,[\omega|_F])/\mathrm{Aut}^0(F)$ is finite.
Taking a finite unramified covering of $Y$ we can then assume that the structure group reduces to $\mathrm{Aut}^0(F)$.

If we are in the case when $b_1(F)=0$, then we have that $\mathrm{Aut}^0(F)=\{1\}$
and so the structure group of the bundle is trivial, and the bundle is holomorphically trivial.

We now treat the case when $F$ is a torus and $b_1(Y)=0$. Since $Y$ is Calabi-Yau, the Beauville-Bogomolov-Calabi decomposition theorem \cite{Be, Bo, Cal} shows that $\pi_1(Y)$ is finite, so up to a
finite unramified covering of $Y$ we can assume that $Y$ is simply connected.
Since $F$ is a torus, we have that $\mathrm{Aut}^0(F)\cong F$ and so $f:X\to Y$ is a principal torus bundle.
If we call $\mathcal{F}$ the sheaf over $Y$ of germs of local holomorphic maps from $Y$ to $F$, then principal bundles over $Y$ with fiber $F$ are classified by
$H^1(Y,\mathcal{F})$, and we will write $\zeta\in H^1(Y,\mathcal{F})$ for the class of our bundle. If we write $F=\mathbb{C}^n/\Lambda$ with $\Lambda\cong \mathbb{Z}^{2n}$, then the short exact sequence
$$0\to\Lambda\to\mathbb{C}^n\to F\to 0,$$ gives a long exact sequence
$$H^1(Y,\Lambda)\to H^1(Y,\mathcal{O})^n\to H^1(Y,\mathcal{F})\overset{c}{\to}H^2(Y,\Lambda)\to H^2(Y,\mathcal{O})^n,$$
where we call $c(\zeta)$ the Chern class of the bundle. $Y$ is simply connected implies $H^1(Y,\mathcal{O})^n=0$. On the other hand, the
fact that $X$ is K\"ahler implies that $c(\zeta)$ is a torsion class, thanks to a theorem of Blanchard \cite{Bl} (see also \cite[Theorem 1.7]{Ho}). But $H^2(Y,\Lambda)$ is torsion-free since $H_1(Y,\mathbb{Z})=0$, and hence $c(\zeta)=0$, and the principal bundle is holomorphically trivial.
\end{proof}

\begin{rk}
If we assume that $f:X\to Y$ is a holomorphic fiber bundle with base $Y$ and fiber $F$ Calabi-Yau manifolds, with $b_1(Y)=0$, but $F$ not necessarily a torus, then we can conclude that
$f$ is equal to the composition of a holomorphic fiber bundle with Calabi-Yau fiber over $Y\times T$ where $T$ is a torus, composed with the trivial bundle $Y\times T\to Y$.

Indeed, Campana \cite{Ca2} and Fujiki \cite{Fu2} constructed a relative Albanese map for $f:X\to Y$ which is a commutative diagram of compact complex manifolds and holomorphic maps
\[
\xymatrix{
X\ar[rr]^{g}\ar[dr]_{f} & & \mathrm{Alb}(X/Y)\ar[dl]^{A}\\
 & Y & }
\]
where $A:\mathrm{Alb}(X/Y)\to Y$ is a smooth submersion with fiber $A^{-1}(y)\cong \mathrm{Alb}(X_y)$ the Albanese torus of $X_y=f^{-1}(y)\cong F$, and with $g|_{X_y}:X_y\to A^{-1}(y)$
isomorphic to the Albanese map $A_y:X_y\to \mathrm{Alb}(X_y)$ (for all $y\in Y$). Now all fibers $X_y$ are biholomorphic to $F$, and so their Albanese tori are all biholomorphic to $\mathrm{Alb}(F)$, and therefore $A$ is a holomorphic fiber bundle by the Fischer-Grauert theorem \cite{FG}.
Since the fiber $X_y\cong F$ is Calabi-Yau, its Albanese map $A_y$ is a holomorphic fiber bundle with connected fiber $F_1$ which is Calabi-Yau, thanks to \cite{Cal}.

We claim now that $g$ is a submersion. Indeed, pick any point $z\in \mathrm{Alb}(X/Y)$, let $y=A(z)$, and let $V$ be any tangent vector to $\mathrm{Alb}(X/Y)$ at $z$.
Then $A_*V=0$ if and only if $V$ is tangent to $\mathrm{Alb}(X_y)$ (the fiber of $A$). Since $f$ is a submersion, we can find a vector $W$ tangent to $X$ at some point $x$ in the fiber $X_y$ such that $f_*W=A_*V$. Hence $A_*(g_*W-V)=0$, i.e. $g_*W-V$ is tangent to $\mathrm{Alb}(X_y)$.
Since $g|_{X_y}$ is isomorphic to $A_y:X_y\to \mathrm{Alb}(X_y)$ which is a bundle, there is a vector $Z$ tangent to $X$ at $x\in X_y$ with $g_*Z=g_*W-V,$ and so $g_*(W-Z)=V$ and $g$ is a submersion. In particular, since $X$ is K\"ahler, we see that $\mathrm{Alb}(X/Y)$ is K\"ahler too.

Since all fibers of $g$ are biholomorphic to $F_1$, the Fischer-Grauert theorem implies that $g$ is a holomorphic fiber bundle. Now $A:\mathrm{Alb}(X/Y)\to Y$
is a holomorphic fiber bundle with total space K\"ahler, $Y$ simply connected, and fiber a torus, hence this bundle must be trivial thanks to Theorem \ref{bund}.
We thus obtain a new diagram
\[
\xymatrix{
X\ar[rr]^{g'}\ar[dr]_{f} & & Y\times T\ar[dl]^{\pi_Y}\\
 & Y & }
\]
where $T$ is a torus and $g'$ is still a holomorphic fiber bundle with fiber $F_1$ Calabi-Yau.
\end{rk}


\begin{thebibliography}{99}
\bibitem{A} F. Ambro, {\em The moduli $b$-divisor of an lc-trivial fibration}, Compos. Math. {\bf 141} (2005), no. 2, 385--403.
\bibitem{Be} A. Beauville, {\em Vari\'et\'es K\"ahleriennes dont la premi\`ere classe de Chern est nulle}, J. Differential Geom. {\bf 18} (1983), no. 4, 755--782.
\bibitem{Ber} B. Berndtsson, {\em Curvature of vector bundles associated to holomorphic fibrations}, Ann. of Math. (2) {\bf 169} (2009), no. 2, 531--560.
\bibitem{Bo} F.A. Bogomolov, {\em K\"ahler manifolds with trivial canonical class}, Math. USSR Izv. {\bf 8} (1974), no. 1, 9--20.
\bibitem{BDPP} S. Boucksom, J.-P. Demailly, M. P\u{a}un, T. Peternell, {\em The pseudo-effective cone of a compact K\"ahler manifold and varieties of negative Kodaira dimension}, J . Algebraic Geom. {\bf 22} (2013), 201--248.
\bibitem{Bl} A. Blanchard, {\em Sur les vari\'et\'es analytiques complexes}, Ann. Sci. \'Ecole Norm. Sup. (3) {\bf 73} (1956), 157--202.
\bibitem{BG} R.L. Bryant, P.A. Griffiths, {\em Some observations on the infinitesimal period relations for regular threefolds with trivial canonical bundle}, in {\em Arithmetic and geometry, Vol. II}, 77--102, Progr. Math., 36, Birkh\"auser, Boston, 1983.
\bibitem{Cal} E. Calabi, {\em On K\"ahler manifolds with vanishing canonical class}, in {\em Algebraic geometry and topology. A symposium in honor of S. Lefschetz},  pp. 78--89. Princeton University Press, Princeton, N. J., 1957.
\bibitem{Ca2} F. Campana, {\em R\'eduction d'Alban\`ese d'un morphisme propre et faiblement k\"ahl\'erien. I}, Compositio Math. {\bf 54} (1985), no. 3, 373--398.
\bibitem{CMP} J. Carlson, S. M\"uller-Stach, C. Peters, {\em Period mappings and period domains},  Cambridge University Press, 2003.
\bibitem{COP} F. Catanese, K. Oguiso, T. Peternell, {\em On volume-preserving complex structures on real tori}, Kyoto J. Math. {\bf 50} (2010), no. 4, 753--775.
\bibitem{CCL} Z. Chen, S.-Y. Cheng, Q. Lu, {\em On the Schwarz lemma for complete K\"ahler manifolds}, Sci. Sinica {\bf 22} (1979), no. 11, 1238--1247.
\bibitem{De} P. Deligne, {\em Travaux de Griffiths}, S\'em. Bourbaki Exp. 376, 213–-237, Lecture Notes in Math., 180, Springer, Berlin, 1970.
\bibitem{EF} G. Elencwajg, O. Forster, {\em Vector bundles on manifolds without divisors and a theorem on deformations}, Ann. Inst. Fourier (Grenoble) {\bf 32} (1982), no. 4, 25--51.
\bibitem{FL} H. Fang, Z. Lu, {\em Generalized Hodge metrics and BCOV torsion on Calabi-Yau moduli}, J. Reine Angew. Math. {\bf 588} (2005), 49--69.
\bibitem{Fi} J. Fine, {\em Fibrations with constant scalar curvature K\"ahler metrics and the CM-line bundle}, Math. Res. Lett. {\bf 14} (2007), no. 2, 239--247.
\bibitem{FG} W. Fischer, H. Grauert, {\em Lokal-triviale Familien kompakter komplexer Mannigfaltigkeiten}, Nachr. Akad. Wiss. G\"ottingen Math.-Phys. Kl. II (1965), 89--94.
\bibitem{Fu} A. Fujiki, {\em On automorphism groups of compact K\"ahler manifolds}, Invent. Math. {\bf 44} (1978), no. 3, 225--258.
\bibitem{Fu2} A. Fujiki, {\em Relative algebraic reduction and relative Albanese map for a fiber space in $\mathcal{C}$}, Publ. Res. Inst. Math. Sci. {\bf 19} (1983), no. 1, 207--236.
\bibitem{FS} A. Fujiki, G. Schumacher, {\em The moduli space of extremal compact K\"ahler manifolds and generalized Weil-Petersson metrics},  Publ. Res. Inst. Math. Sci.  26  (1990),  no. 1, 101--183.
\bibitem{FG2} O. Fujino, Y. Gongyo, {\em On images of weak Fano manifolds},  Math. Z. {\bf 270} (2012), no. 1-2, 531--544.
\bibitem{FM} O. Fujino, S. Mori, {\em A canonical bundle formula}, J. Differential Geom. {\bf 56} (2000), no. 1, 167--188.
\bibitem{GM} W.M. Goldman, J.J. Millson, {\em The deformation theory of representations of fundamental groups of compact K\"ahler manifolds}, Inst. Hautes \'Etudes Sci. Publ. Math. {\bf 67} (1988), 43--96.
\bibitem{Gr2} P.A. Griffiths, {\em  Periods of integrals on algebraic manifolds. II. Local study of the period mapping}, Amer. J. Math. {\bf 90} (1968), 805--865.
\bibitem{Gr3} P.A. Griffiths, {\em  Periods of integrals on algebraic manifolds. III. Some global differential-geometric properties of the period mapping}, Inst. Hautes \'Etudes Sci. Publ. Math. {\bf 38} (1970), 125--180.
\bibitem{Gr} P.A. Griffiths, ed., {\em Topics in transcendental algebraic geometry,}  Annals of Mathematics Studies, 106. Princeton University Press, Princeton, NJ, 1984.
\bibitem{GS} P.A. Griffiths, W. Schmid, {\em Locally homogeneous complex manifolds}, Acta Math. {\bf 123} (1969), 253--302.
\bibitem{GTZ} M. Gross, V. Tosatti, Y. Zhang, {\em Collapsing of abelian fibred Calabi-Yau manifolds}, Duke Math. J. {\bf 162} (2013), no. 3, 517--551.
\bibitem{GTZ2} M. Gross, V. Tosatti, Y. Zhang, {\em Gromov-Hausdorff collapsing of Calabi-Yau manifolds}, arXiv:1304.1820.
\bibitem{Ho} T. H\"ofer, {\em Remarks on torus principal bundles}, J. Math. Kyoto Univ. {\bf 33} (1993), no. 1, 227--259.
\bibitem{Ko} K. Kodaira, {\em  Complex manifolds and deformation of complex structures}, Springer-Verlag, Berlin, 2005.
\bibitem{KL} J. Koll\'ar, M. Larsen, \emph{Quotients of Calabi-Yau varieties}, in {\em Algebra, arithmetic, and geometry: in honor of Yu. I. Manin. Vol. II}, 179--211,
Progr. Math., 270, Birkh\"auser Boston, Inc., Boston, MA, 2009.
\bibitem{Li} D.I. Lieberman, {\em  Compactness of the Chow scheme: applications to automorphisms and deformations of K\"ahler manifolds}, 140--186, Lecture Notes in Math., 670, Springer, Berlin, 1978.
\bibitem{Lu} Z. Lu, {\em On the geometry of classifying spaces and horizontal slices},  Amer. J. Math. {\bf 121} (1999), no. 1, 177--198.
\bibitem{Ro} H.L. Royden, {\em The Ahlfors-Schwarz lemma in several complex variables}, Comment. Math. Helv. {\bf 55} (1980), no. 4, 547--558.
\bibitem{ST} J. Song, G. Tian {\em Canonical measures and K\"ahler-Ricci flow}, J. Amer. Math. Soc. {\bf 25} (2012), no. 2, 303--353.
\bibitem{Tj} G.N. Tjurina, {\em Deformation of complex structures on algebraic varieties}, Soviet Math. Dokl. {\bf 4} (1963), 1567--1574.
\bibitem{To1} V. Tosatti, {\em A general Schwarz Lemma for almost-Hermitian manifolds}, Comm. Anal. Geom. {\bf 15} (2007), no.5, 1063-1086.
\bibitem{To} V. Tosatti, {\em  Adiabatic limits of Ricci-flat K\"ahler metrics}, J. Differential Geom. {\bf 84} (2010), no. 2, 427--453.
\bibitem{TWY} V. Tosatti, B. Weinkove, X. Yang, {\em Collapsing of the Chern-Ricci flow on elliptic surfaces}, arXiv:1302.6545.
\bibitem{VZ} E. Viehweg, K. Zuo, {\em On the isotriviality of families of projective manifolds over curves}, J. Algebraic Geom. {\bf 10} (2001), no. 4, 781--799.
\bibitem{Yau} S.-T. Yau, {\em A general Schwarz lemma for K\"{a}hler manifolds}, Amer. J. Math.  {\bf 100} (1978), 197--204.
\bibitem{Y} S.-T. Yau, {\em On the Ricci curvature of a compact K\"ahler manifold and the complex Monge-Amp\`ere equation, I}, Comm. Pure Appl. Math. {\bf 31} (1978), no.3, 339--411.
\bibitem{ZZ} Y. Zhang, K. Zuo, {\em Calabi-Yau varieties with semi-stable fibre structures}, Manuscripta Math. {\bf 126} (2008), no. 3, 273--291.
\end{thebibliography}
 \end{document}